\documentclass[reqno, 11pt]{amsart} 
\usepackage{amsfonts, amsmath, amssymb, amsthm}
\usepackage[margin=2.65cm, heightrounded]{geometry}
\usepackage{fancybox}
\usepackage{hhline, float}
\usepackage{mathrsfs}
\usepackage{nccmath}
\usepackage{bm}
\usepackage{ulem}
\usepackage{cite}
\usepackage[dvipsnames]{xcolor}
\usepackage[colorlinks=true]{hyperref}
\hypersetup{citecolor=red, linkcolor=blue}

\usepackage{hhline,float}
\usepackage{enumitem}
\usepackage{microtype}

\allowdisplaybreaks
\numberwithin{equation}{section}

\newtheorem{thm}{Theorem}[section]
\newtheorem{lemma}[thm]{Lemma}
\newtheorem{prop}[thm]{Proposition}
\newtheorem{cor}[thm]{Corollary}

\theoremstyle{definition}
\newtheorem{rmk}[thm]{Remark}
\newtheorem{defn}[thm]{Definition}

\newcommand{\pa}{\partial}

\newcommand{\R}{\mathbb{R}}

\newcommand{\e}{\textup{e}}

\newcommand{\dx}{\textup{d}x}

\newcommand{\calR}{\mathcal{R}}

\makeatletter
\@namedef{subjclassname@2020}{\textup{2020} Mathematics Subject Classification}
\makeatother

\begin{document}
\title[Anisotropic N-Laplacian equation]
{Rigidity of weak solutions for anisotropic N-Laplacian equation with Neumann or Robin boundary condition}

\author{Yuxia Guo}
\address[Yuxia Guo]{Department of Mathematical Sciences, Tsinghua University, Beijing, P.R. China}
\email{yguo@tsinghua.edu.cn}

\author{Yichen Hu}
\address[Yichen Hu]{School of Mathematical Sciences, Dalian University of Technology, Dalian, Liaoning, P. R. China}
\email{huyc24@dlut.edu.cn}

\author{Shaolong Peng}
\address[Shaolong Peng]{School of Mathematical Sciences, Beihang University (BUAA), Beijing, P.R. China}
\email{slpeng@amss.ac.cn}

\author{Tingfeng Yuan}
\address[Tingfeng Yuan]{Department of Mathematical Sciences, Tsinghua University, Beijing, P. R. China}
\email{ytf22@mails.tsinghua.edu.cn}

\begin{abstract}
This paper is devoted to the rigidity of weak solutions for anisotropic $N$-Laplacian equations with Neumann or Robin boundary conditions on smooth bounded convex domains of $\mathbb{R}^N$. The anisotropic operator is given by $$a(\xi) = H^{N-1}(\xi)\nabla H(\xi),$$
where $H$ stands for a norm on $\mathbb{R}^N$; this formulation contains the classical $N$-Laplacian as a special case. We establish a key integral inequality involving the anisotropic gradient and the second fundamental form of the domain boundary, which acts as the core technical tool in our proofs. Under natural monotonicity assumptions on the nonlinearity, we prove that all weak solutions to the Neumann boundary problem are constant, without requiring any a priori boundedness assumption on the solution. Furthermore, we extend this rigidity result to Robin boundary value problems by imposing suitable constraints on the boundary nonlinear term. Moreover, our rigidity results remain valid not only on bounded convex domains but also on suitable unbounded domains. By working under substantially weaker assumptions than those previously available, we establish rigidity results that fill the gaps in the existing literature for anisotropic $N$-Laplacian equations with nonlinear boundary conditions and substantially extend the rigidity theory of anisotropic quasilinear elliptic equations at the critical exponent $p=N$.
\end{abstract}

\date{}
\maketitle

\section{Introduction and background}

Nonlinear elliptic equations involving the $p$-Laplacian have been studied for a long time. They arise naturally in nonlinear elasticity, geometry, reaction-diffusion models and variational problems. A typical example is
\begin{align}\label{p-Laplaican}
    -\Delta_p u=f(u), \ \ \text{in} \ \ \Omega \subset \R^N,
\end{align}
where $\Delta_p u=\operatorname{div}(|\nabla u|^{p-2}\nabla u)$ is the $p$-Laplacian operator, $\Omega\subset \R^N$ is a smooth bounded domain. The qualitative behavior of solutions is closely related to the growth of the nonlinearity $f$. When $1<p<N$, the Sobolev critical exponent $p^*=\frac{Np}{N-p}$ provides the natural critical growth. However, the limiting case $p=N$ is different because there is no finite Sobolev critical exponent, and the corresponding critical growth is of the exponential type, as described by the Trudinger-Moser inequality, see \cite{S1071, T1967}. Thus exponential nonlinearities are natural in the study of equation involving $N$-Laplacian.

In this paper, we will focus on the rigidity results of the solutions for \eqref{p-Laplaican}. In the case of  $p=2$, equation \eqref{p-Laplaican} reduces to the Laplace equation, Gidas and Spruck \cite{GS1981} proved that \eqref{p-Laplaican} ($p=2$) has no non-negative nontrivial solution when $N \geq 3$, $\Omega = \R^N$, $f(u) = u^q$ with $1< q < \frac{N+2}{N-2}$. When $q=\frac{N+2}{N-2}$, Gidas, Ni and Nirenberg \cite{GNN1979} studied the symmetric property of solutions and classified all of the positive solutions under an additional assumption that $u(x) = O(|x|^{N-2})$ at infinity. Later, by employing the Kelvin transformation, Caffarelli, Gidas and Spruck \cite{CGS1989} removed this assumption and proved the same result. On the other hand, if $\Omega$ is star-shaped, $u$ satisfies the Dirichlet boundary condition on $\pa\Omega$ and $f(u) = u^q$ with $q \geq \frac{N+2}{N-2}$, then the Pohozaev identity yields the non-existence result of the non-negative solution for \eqref{p-Laplaican} with $p=2$, see \cite{P1965}.

For $p$-Laplacian type equations with $1<p<N$, Liouville type theorems for \eqref{p-Laplaican} also have a long history. We mention the works of Mitidieri and Pohozaev \cite{MP1999,MP2001}, Bidaut-V\'{e}ron and Pohozaev \cite{BP2001}, Birindelli and Demengel \cite{BD2002}, and Serrin and Zou \cite{SZ2002}. These works show that the existence or nonexistence of nontrivial entire solutions is closely related to the exponent $p$, the dimension $N$, and the growth exponent of the nonlinearity $f(u)$. In particular, Serrin and Zou \cite{SZ2002} established Cauchy-Liouville type theorems and universal boundedness results for a broad class of quasilinear elliptic equations and inequalities. Dai et al. \cite{DDGL2026} generalized the existence results of positive nonradial solutions for the H\'{e}non equation:
 $$-\Delta u=(N+\alpha)(N-2)|x|^\alpha u^{\frac{N+2+2\alpha}{N-2}} \quad \hbox{in } \mathbb{R}^N, \alpha>0, N\geq 3,$$
 to the general $p$-Laplacian, and analyzed bifurcation phenomena at critical parameter values via approximation and topological arguments. Besides, for critical $p$-Laplacian equations, the rigidity  results are related to the extremals of the Sobolev inequality. Positive finite-energy solutions of \eqref{p-Laplaican} in $\R^N$ with $f(u) = u^{p^* -1}$ are classified as Sobolev bubbles, see \cite{DMMS2014, V2016, S2016}. Rigidity  results or classification without finite-energy assumptions, as well as anisotropic results in convex cones, have also been studied in \cite{CMR2023, O2025, CFR2020}. Furthermore, in a bounded convex domain, the non-existence of positive bounded solution of \eqref{p-Laplaican} with Neumann or Robin boundary condition was established in \cite{CCR2021}.

In the limiting case of $p=N$, the Sobolev critical power is replaced by the Trudinger--Moser exponential growth. Elliptic problems involving the $N$-Laplacian and exponential nonlinearities have been studied both on bounded domains and in the whole space, see, for example, \cite{A1990, do1996, SS1999, LL2012, deF2022} and references therein. In particular, if $f(u) = e^u$ and $\Omega = \R^N$, the finite-mass solutions of \eqref{p-Laplaican} were classified in
\cite{E2018}. Related anisotropic classification results for the Finsler $N$-Laplacian were obtained in \cite{CL2024}. Recently, half-space problems with nonlinear Neumann boundary conditions have also been considered in \cite{DGHP2026}.

For the subcritical anisotropic $p$-Laplacian equations with $1<p<N$, the classification and non-existence of non-constant weak solutions under Neumann and Robin boundary conditions have been well established in the literature. In particular, the seminal work by Ciraolo et al. \cite{CCR2021} systematically investigated such problems on convex domains. By establishing a delicate integral inequality associated with the anisotropic gradient and the second fundamental form of the domain boundary, they proved that all positive bounded weak solutions to the Neumann problem are constant under suitable monotonicity assumptions on the nonlinearity. Furthermore, they extended this rigidity result to problems with general nonlinear Robin boundary terms under additional structural constraints. Their main approach relies on integral identities specifically constructed for the subcritical regime, which have become a fundamental tool for related studies.

Despite the rich achievements for $1<p<N$, \textit{rigidity results (including Liouville-type theorems and classification results)  for the anisotropic $N$-Laplacian with general nonlinearities under Neumann and nonlinear Robin boundary conditions on bounded convex domains remain largely unexplored}. Most existing researche on the anisotropic $N$-Laplacian (or $p$-Laplacian) mainly focus on spectral theory \cite{PPS}, existence and multiplicity of nontrivial solutions via variational methods  \cite{FH, GL2005}, or explicit solution classification for exponential nonlinearities in the whole space \cite{CL2024} and unbounded convex cones \cite{DGL}. Very recently, Chen et al. \cite{CDGL2026} extended the classical Liouville-type results for $p$-Laplacian equations from $\mathbb{R}^N$ to general convex cones. Without imposing the finite-energy condition, they proved the triviality of all nonnegative solutions in the subcritical case and established a complete classification of positive solutions for the critical problem. For a class of weighted $N$-Liouville equations, Ciraolo, Esposito and Li \cite{CEL} developed a new analytical framework based on the $P$-function method. They derived the explicit form of all solutions for $-1<\alpha\leq 0$, and further characterized the degeneracy of radial solutions at some special parameter values.
The integral framework developed for subcritical $p$-Laplacian equations cannot be directly adapted to the critical case $p=N$, due to the breakdown of the Sobolev conjugate exponent and singular transformation formulas. Moreover, \textit{there are almost no available conclusions concerning the constant solution property for weak solutions to nonlinear Robin boundary value problems involving the anisotropic $N$-Laplacian}.

 The main purposes of present paper are to solve the above problems and establish new classification results for the anisotropic $N$-Laplacian equations with Neumann and Robin boundary conditions. Utilizing a newly derived integral inequality that couples anisotropic gradient quantities and the boundary second fundamental form, we prove that all weak solutions to the Neumann problem must be constant under reasonable conditions on the nonlinear term. We further extend this rigidity conclusion to Robin boundary value problems by imposing proper assumptions on the boundary nonlinearity. Our work fills the above research gaps and complements the existing theory for anisotropic quasilinear elliptic operators.

More precisely, in the following of the  paper,  we will consider the rigidity result for the following anisotropic $N$-Laplacian equation with Neumann or Robin type boundary conditions. Let $\Omega\subset\mathbb{R}^N$ be a smooth bounded convex domain. We consider  the following problem:
\begin{align}\label{Main}
    \begin{cases}
        \operatorname{div}(a(\nabla u))+f(u)=0,&\text{in }\Omega,\\
        a(\nabla u)\cdot\nu+h(u)=0,&\text{on }\partial\Omega,
    \end{cases}
\end{align}
where
   \begin{align*}
         a(\xi)=H^{N-1}(\xi)\nabla H(\xi),
   \end{align*}
and $H$ is a norm in $\mathbb{R}^N$. When $H(\xi)=|\xi|$, the operator becomes the standard $N$-Laplacian. Besides, the boundary term $h(u)$ allows us to treat both pure Neumann and Robin type boundary conditions. First, we define the weak solution of equation \eqref{Main}:

\begin{defn}[Weak solution]\label{def:weak-solution-main}
We say that $u\in W^{1,N}(\Omega)$ is a weak solution of \eqref{Main} if $f(u)\in L^1(\Omega)$, $h(u)$ is integrable on $\partial\Omega$, and
\begin{align}\label{weak-form}
        \int_{\partial\Omega} h(u)\phi\,d\sigma
        +\int_{\Omega} a(\nabla u)\cdot\nabla\phi\,dx
        =\int_{\Omega} f(u)\phi\,dx,
        \qquad \forall\phi\in C^1(\overline\Omega).
\end{align}
\end{defn}

Our first result concerns the pure Neumann problem:

\begin{thm}\label{Thm1}
Let $\Omega\subset\R^N$ be bounded, connected, convex and of class $C^2$.  Let $H$ be a $C^2$ strictly convex norm and let $f\in C^1(\R)$ satisfy
\begin{align}\label{Thm1-assumption}
        \Phi'(t)\le 0 \ \ \text{for all }\ t\in\R, \ \ \Phi(t)=e^{-t}f(t).
\end{align}
If $u$ is a weak solution of
\begin{equation}\label{Thm1-equ}
    \begin{cases}
        \operatorname{div}(a(\nabla u))+f(u)=0, &\text{in }\Omega,\\
        a(\nabla u)\cdot\nu=0, &\text{on }\partial\Omega,
    \end{cases}
\end{equation}
then $u$ is constant.
\end{thm}

As a consequence of Theorem \ref{Thm1}, we have the following Corollary.

\begin{cor}\label{Thm5}
    Let $\Omega$ be a smooth bounded convex domain, let $H$ be a $C^2$ strictly convex norm, and let $f\in C^1(\R)$ satisfy \eqref{Thm1-assumption}.  If $u$ is a weak solution of
    \begin{align}\label{Thm5-equ}
        \begin{cases}
            \operatorname{div}(a(\nabla u))+f(u)=0, & \text{in }\Omega,\\
            |\nabla u|=0, & \text{on }\partial\Omega,
        \end{cases}
    \end{align}
    then $u$ is constant.
\end{cor}

Next, we consider the case when $h \neq 0$. In this case, we will deal with a boundary integral involving the second derivative of $u$. However, the regularity theory for anisotropic $N$-Laplacian equations does not guarantee $C^2$ regularity up to the boundary, due to the possible degeneracy at critical points. Consequently, the second-order boundary quantities appearing in our integral argument are not automatically well-defined for an arbitrary weak solution.

The following definition is therefore introduced as a technical admissibility condition ensuring that the regularization procedure and the boundary identities used in this paper are rigorously justified. It is not intended as a new intrinsic notion of solution, but rather as a convenient framework for the method developed below.

\begin{defn}\label{def:classicalRA}
A solution $u$ of \eqref{Main} is called a classical solution if $u\in C^2(\Omega)\cap C^1(\overline\Omega)$, $a(\nabla u)\in C^1(\overline\Omega;\R^N)$, and \eqref{Main} holds pointwise.
\end{defn}

\begin{defn}\label{def:weakRA}
A weak solution $u$ of \eqref{Main} with $u \in W^{1,N}(\Omega)\cap C^1(\overline\Omega) $ is called an admissible weak solution if there exist smooth uniformly elliptic regularizations $a_j,\ f_j,\ h_j$, and the corresponding classical solutions $u_j$ of \eqref{Main} on $\Omega$ with $a,f,h$ been replaced by $a_j,f_j,h_j$,  such that
\[
    u_j\longrightarrow u \ \ \text{in} \ \ C^1(\overline\Omega).
\]
The regularized data $a_j, f_j, h_j$ are assumed to converge locally uniformly to the original data $a,f,h$ on the bounded sets attained by $u_j$ and $\nabla u_j$.
\end{defn}

\begin{rmk}
If, in addition,
\[
    |\nabla u|\ge c_0>0 \ \ \text{on}\ \overline\Omega,
\]
and $H,f,h,\Omega$ are sufficiently smooth, then the
equation is uniformly elliptic along $u$. Standard boundary
regularity theory yields
\[
    u\in C^{2,\gamma}(\overline\Omega), \qquad a(\nabla u) \in C^{1,\gamma}(\overline\Omega;\mathbb R^N).
\]
In this case, the solution $u$ is classical.
\end{rmk}

In order to prove our main theorem, we need the following Proposition.

\begin{prop}\label{Prop2}
Let $u$ be a classical solution of \eqref{Main}. Then
\begin{align}\label{prop2-corrected}
        \frac{N-1}{N}\int_{\Omega}e^{u/N}H^N(\nabla u)\Phi'(u)\,dx
        \ge \int_{\partial\Omega}B_\Omega[u]d\sigma,
\end{align}
where $B_\Omega[u]$ is defined as
\begin{align}\label{eq:BOmega}
    B_{\Omega}[u] = \e^{-\frac{(N-1)u}{N}} \bigg( & \frac{(N-1)^2}{N^2} H^N(\nabla u) a(\nabla u) \cdot \nu - \nabla (a(\nabla u) \cdot \nu) a(\nabla u ) \nonumber \\
    & - \frac{1}{N} a(\nabla u) \cdot \nu f(u) + \Pi_x(a^T(\nabla u), a^T(\nabla u) ) \bigg),
\end{align}
with $a^T(\nabla u) = a(\nabla u) - (a(\nabla u) \cdot  \nu) \nu$.
If equality holds in \eqref{prop2-corrected}, then either
    \begin{align}\label{affine}
        u(x) = N \log N - N \log (\ell \cdot x + c),
    \end{align}
    with $\ell \in \R^N, c \in \R$ chosen such that $\ell \cdot x + c > 0 $ in $\Omega$, or
    \begin{align}\label{bubble-form-intro}
        u(x) = N \log N - N \log (c_1 + c_2 H_0(x_0-x)^{\frac{N}{N-1}}),
    \end{align}
    with $c_2 \neq 0$, $x_0 \in \R^N$ and $c_1, c_2$ be chosen such that $c_1 + c_2 H_0(x_0-x)^{\frac{N}{N-1}} $ in $\Omega$.
\end{prop}

By virtue of Proposition \ref{Prop2}, we can prove a more general rigidity result (i.e.  classification theorem) with Robin boundary condition.

\begin{thm}\label{Thm3}
    Let $\Omega\subset\R^N$ be bounded, connected, convex and of class $C^2$.  Let $f,h\in C^1(\R)$ and assume \eqref{Thm1-assumption} holds.  Let $u$ be an admissible weak solution of \eqref{Main}.  If $u$ satisfies the sign condition
    \begin{align}\label{B-Omega}
        \int_{\pa \Omega} B_{\Omega}[u] \ \text{d} \sigma \geq 0,
    \end{align}
    where $B_{\Omega}[u]$ is defined in \eqref{eq:BOmega}. Then either
    \begin{align}\label{linear-log}
        u(x) = N \log N - N \log (\ell \cdot x + c),
    \end{align}
    with $\ell \in \R^N, c \in \R$ chosen such that $\ell \cdot x + c > 0 $ in $\Omega$, or
    \begin{align}\label{bubble-log}
        u(x) = N \log N - N \log (c_1 + c_2 H_0( x_0 -x )^{\frac{N}{N-1}}),
    \end{align}
    with $c_2 \neq 0$, $x_0 \in \R^N$ and $c_1, c_2$ be chosen such that $c_1 + c_2 H_0(x_0-x)^{\frac{N}{N-1}} $ in $\Omega$.
\end{thm}

Then, a natural question is: under what conditions does
\eqref{B-Omega} hold? It is worth mentioning that when $\Omega$ is convex, then the first and fourth term in \eqref{eq:BOmega} are nonnegative. The main difficulty comes from the second term, which involves $\nabla (a(\nabla u) \cdot \nu)$. The Robin boundary condition gives $a(\nabla u) \cdot \nu = -h(u)$ on $\pa \Omega$. However, this identity holds only on the boundary and can therefore be differential only in the tangential direction $\tau$. More precisely, on $\partial\Omega$, we have
\begin{align*}
    \nabla  = \pa_{\nu} \cdot \nu+ \nabla_{\pa \Omega} \cdot \tau,
\end{align*}
where $\nabla_{\pa \Omega}$ is the derivative along $\pa \Omega$. Thus, the Robin condition gives
\[
    \nabla_{\partial\Omega} \bigl(a(\nabla u)\cdot\nu\bigr) = -h'(u)\nabla_{\partial\Omega}u,
\]
but it does not determine the normal derivative $ \partial_\nu\bigl(a(\nabla u)\cdot\nu\bigr)$. In Section 4, we use the decomposition \eqref{boundary-div-decomposition} to express this normal derivative
in terms of $\operatorname{div}a(\nabla u)$, the tangential
divergence, and the mean curvature $\mathcal H_{\partial\Omega}$. Combining this decomposition with
integration by parts on $\partial\Omega$, we obtain an equivalent expression for $\int_{\partial\Omega} B_\Omega[u]\ \text{d}\sigma$, see Lemma \ref{lem:boundary-reduction} for more detail. This representation then allows us to give some sufficient conditions on $H$, $f$, and $h$, under which \eqref{B-Omega} holds.

\begin{cor}\label{cor4}
    Let the assumptions of Theorem \ref{Thm3} hold. Suppose one of the following conditions holds:
    \begin{enumerate}[label=(\roman*)]
        \item Let $H_{\pa \Omega}(x)$ be the mean curvature at $x \in \pa \Omega$ and $H_* = \min_{x \in \pa \Omega} H_{\pa \Omega}(x)$, with the following conditions hold:
        \begin{align*}
               h(s) \leq 0, \ \ 2 h'(s) = \frac{N-1}{N}h(s), \ \ (N-1)H_* h(s)^2 - \frac{N-1}{N} h(s) f(s) \geq 0.
        \end{align*}
        \item Suppose $H(\xi)=|\xi|$ with the following conditions hold:
         \begin{align*}
             & h(s) \leq 0, \ \ 2 h'(s) - \frac{(N-1)(2N-1)}{N^2}h(s) \geq 0, \\
             & (N-1)H_* h(s)^2 - \frac{N-1}{N} h(s) f(s) \geq 0.
        \end{align*}
        \item Suppose $H(\xi)=|\xi|$ with the following conditions hold:
         \begin{align*}
             & h(s) \geq 0, \ \ \ 2 h'(s) - \frac{(N-1)(2N-1)}{N^2}h(s) \geq 0, \\
             & (N-1)H_* h(s)^2 - \frac{N-1}{N} h(s) f(s) - \frac{(N-1)^2}{N^2} h(s)^{\frac{2N-1}{N-1}} \geq 0.
        \end{align*}
    \end{enumerate}
    Then the boundary condition \eqref{B-Omega} is satisfied.
\end{cor}

Finally, we state the corresponding result on an unbounded convex domain.

\begin{thm}\label{thm:unbounded}
Let $\Omega\subset\mathbb R^N$ be connected, convex, $C^2$, and unbounded. Let $u$ be a weak solution of
\eqref{Main}, which locally satisfies the condition of admissible weak solution. Assume further  there is a sequence of bounded convex domains $\{D_j\}$ converges to $\Omega$ and
\begin{align}\label{unbounded-assumption}
    \int_\Omega e^{u/N}H^N(\nabla u)|\Phi'(u)|\,dx<\infty, \ \ \int_{\pa D_j} \overline{B}_{\Omega}[u] \ \text{d} \sigma \geq 0,
\end{align}
where
\begin{align}\label{eq:overlineBOmega}
    B_{\Omega}[u] = \e^{-\frac{(N-1)u}{N}} \bigg( & \frac{(N-1)^2}{N^2} H^N(\nabla u) a(\nabla u) \cdot \eta_j - \nabla (a(\nabla u) \cdot \eta_j) a(\nabla u )  - \frac{1}{N} a(\nabla u) \cdot \eta_j f(u) \nonumber \\
    & + \Pi_x(a(\nabla u) - (a(\nabla u) \cdot  \eta_j) \eta_j, a(\nabla u) - (a(\nabla u) \cdot  \eta_j) \eta_j)) \bigg),
\end{align}
and $\eta_j$ is the unit normal at $\pa D_j$. If $\Phi'\le0$, Then either
    \begin{align}\label{linear-log+}
        u(x) = N \log N - N \log (\ell \cdot x + c),
    \end{align}
    with $\ell \in \R^N, c \in \R$ chosen such that $\ell \cdot x + c > 0 $ in $\Omega$, or
    \begin{align}\label{bubble-log+}
        u(x) = N \log N - N \log (c_1 + c_2 H_0( x_0 -x )^{\frac{N}{N-1}}),
    \end{align}
    with $c_2 \neq 0$, $x_0 \in \R^N$ and $c_1, c_2$ be chosen such that $c_1 + c_2 H_0(x_0-x)^{\frac{N}{N-1}} $ in $\Omega$.
\end{thm}


The paper is organized as follows. Section \ref{sec:regularity} contains the proof of regularity of weak solution and a sufficient condition for admissible classical solution. In Section \ref{sec:newton}, we prove Proposition \ref{Prop2}. Section \ref{sec:sufficient} proves the main theorems and gives some sufficient conditions for the condition \eqref{B-Omega}.

\section{Regularity and preliminary estimates}\label{sec:regularity}
In this section, we first prove the $L^{\infty}$-bound of weak solution to \eqref{Main}. Before that, we need the following lemma:

\begin{lemma}\label{iteration}\textup{[Lemma 4.1 in \cite{S1965}]}
    Let $\varphi:[k_0,\infty)\to[0,\infty)$ be a non-increasing function. Assume that there exist constants $C>0$, $\alpha>0$, and $\beta>1$ such that
    \begin{align*}
        \varphi(l) \leq \frac{C}{(l-k)^{\alpha}} \varphi(k)^{\beta},
    \end{align*}
    for every $l > k \geq k_0$. Then 
    \begin{align*}
        \varphi(k_0 + d) = 0,
    \end{align*}
    where one may take 
    \begin{align*}
        d  = 2^{\frac{\beta}{\beta -1}} C^{\frac{1}{\alpha}} \varphi(k_0)^{\frac{\beta -1}{\alpha}}.
    \end{align*}
\end{lemma}
Then, we can prove the $L^{\infty}$-bound of weak solution to \eqref{Main} under some constraints.

\begin{lemma}
    Suppose $\Omega \subset \R^N $ is a bounded domain. Given $f,h: \R \rightarrow \R$, where $f$ satisfies
    \begin{align*}
        |f(t)| \leq C \e^{|t|},
    \end{align*}
    for $t \in \R$ and some constant $C>0$. Let $u \in W^{1,N}(\Omega)$ be a weak solution to
    \begin{equation}\label{Main3}
        \begin{cases}
            \textup{div} (a(\nabla u)) + f(u) =0, \ \ &\text{in} \ \ \Omega , \\
            a(\nabla u) \cdot \nu  + h(u)=0, \ \ &\text{on} \ \ \partial \Omega,
        \end{cases}
    \end{equation}
    where $a = H^{N-1}(\cdot)\nabla H(\cdot): \R^N \rightarrow \R^N$ is a continuous vector field. 
    Suppose one of the following conditions holds:\\
    \text{(1)} There is a $p_0 > \frac{N}{N-1}$, such that $h(u) \in L^{p_0}(\partial \Omega)$. \\
    \text{(2)} The solution $u$ is bounded from below and $h(t) \geq 0$, $t \in \R$.\\
    Then, it holds that $u \in L^{\infty}(\Omega)$.
\end{lemma}

\begin{proof}
    Set
    \begin{align*}
        u_{\Omega} = \frac{1}{|\Omega|} \int_{\Omega} u(x) \ \dx. 
    \end{align*}
    Then, it follows from $u \in W^{1,N}(\Omega)$ and  H\"{o}lder inequality that
    \begin{align}\label{u-omega-bound}
        |u_{\Omega}| \leq |\Omega|^{-\frac{1}{N}}||u||_{L^N(\Omega)}. 
    \end{align}
    By  Moser-Trudinger inequality ( see \cite{C2008}),  we have
    \begin{align}\label{MT}
        \int_{\Omega} \exp\left(\frac{\alpha |u(x) - u_{\Omega}|^{\frac{N}{N-1}}}{||\nabla u||_{L^N(\Omega)}^{\frac{N}{N-1}}} \right)  \ \dx \leq C(\Omega) ,
    \end{align}
    where $\alpha$  depends only on $N$ and $\Omega$. 
    
    Next, for any $q \in [1,+\infty)$, we can use Young inequality to obtain
    \begin{align*}
        q|t| \leq  \frac{\alpha|t|^{\frac{N}{N-1}}}{2^{\frac{N}{N-1}}||\nabla u ||_{L^{N}(\Omega)}^{\frac{N}{N-1}}} + C(q, N , \Omega, ||u||_{L^{N}(\Omega)}^{\frac{N}{N-1}}).
    \end{align*}
    Therefore, using $|a+b|^{\frac{N}{N-1}} \leq 2^{\frac{N}{N-1}}(|a|^{\frac{N}{N-1}} + |b|^{\frac{N}{N-1}})$, \eqref{u-omega-bound} and \eqref{MT}, we have
    \begin{align}\label{f-Lq}
       &\quad \int_{\Omega}  |f(u)|^q \ \dx \leq C\int_{\Omega} \e^{q|u|}   \leq e^{C(q,N, \Omega, ||u||_{L^{N}(\Omega)}^{\frac{N}{N-1}})} \int_{\Omega} \exp\left( \frac{\alpha|u(x)|^{\frac{N}{N-1}}}{2^{\frac{N}{N-1}}||\nabla u ||_{L^{N}(\Omega)}^{\frac{N}{N-1}}} \right) \ \dx \nonumber\\
        & \leq C(q,N, \Omega, ||u||_{L^N(\Omega)})\exp\left( \frac{\alpha|u_{\Omega}|^{\frac{N}{N-1}}}{||\nabla u ||_{L^{N}(\Omega)}^{\frac{N}{N-1}}} \right)\int_{\Omega} \exp\left( \frac{\alpha|u(x) - u_{\Omega}|^{\frac{N}{N-1}}}{||\nabla u ||_{L^{N}(\Omega)}^{\frac{N}{N-1}}} \right) \ \dx \nonumber\\
         & \leq C(q,N, \Omega, ||u||_{W^{1,N}(\Omega)}).
    \end{align}

    Now we  define $A_k = \{x \in \Omega: u(x) > k\}$ for $k \geq k_0$, where $k_0$ is defined such that
    \begin{align*}
        |A_{k_0}| = |\{x \in \Omega: u(x) > k_0\}| \leq \frac{1}{2} |\Omega|.
    \end{align*}
    In fact, $k_0$ must exist since
    \begin{align*}
        |A_k| \leq \int_{A_k} \frac{u}{k} \ \dx \leq \frac{|\Omega|^{\frac{N-1}{N}}}{k} ||u||_{W^{1,N}(\Omega)}.
    \end{align*}
    
    Let $v = (u - k)^+$, then $v \in W^{1,N}(\Omega)$ and note that  $k\geq k_0$, we know that the zero set $Z : = \{x \in \Omega: v(x)=0\}$ has positive measure. Thus, by Poincare inequality for functions with positive zero set, we have
    \begin{align}\label{Poincare}
        ||v||_{W^{1,p}(\Omega)} =  ||(u-k)^+||_{W^{1,p}(\Omega)} \leq ||\nabla v||_{L^{p}(\Omega)} = ||\nabla u||_{L^{p}(A_k)},
    \end{align}
    for $1 < p < +\infty$. Test \eqref{Main3} with $v$, we obtain
    \begin{align}\label{u-k}
        \int_{\Omega} a(\nabla u) \cdot \nabla (u-k)^+ \ \dx + \int_{\partial \Omega} h(u) (u-k)^+ \ \text{d}\sigma= \int_{\Omega} f(u) (u-k)^+ \ \dx.
    \end{align}
    First, we suppose $(1)$ holds. 
    Then, by the definition of $a(\cdot)$, we have
    \begin{align}\label{u-k2}
        \int_{\Omega} a(\nabla u) \cdot \nabla (u-k)^+ \ \dx = \int_{A_k} a(\nabla u) \cdot \nabla u  \ \dx \geq c_0 \int_{A_k} |\nabla u|^N \ \dx,
    \end{align}
    where $c_0$ is a positive constant that depends on $H$. Inserting  \eqref{u-k2} into \eqref{u-k}, using trace inequality and \eqref{Poincare}, we obtain
    \begin{align}\label{u-k3}
        c_0 \int_{A_k} |\nabla u|^N \ \dx & \leq \int_{\Omega} f(u) (u-k)^+ \ \dx - \int_{\pa \Omega}h(u)(u-k)^+ \ \text{d} \sigma \nonumber\\
        & \leq  \int_{A_k} |f(u)| (u-k)^+ \ \dx +  \int_{\pa \Omega} |h(u)|(u-k)^+ \ \text{d} \sigma \\
        & \leq ||f(u)||_{L^{p_0}(\Omega)} ||u-k||_{L^{p_0'}(A_k)} + ||h(u)||_{L^{p_0}(\pa \Omega)} ||(u-k)^+||_{L^{p_0'}(\pa \Omega)}, \nonumber \\
        & \leq ||f(u)||_{L^{p_0}(\Omega)} ||u-k||_{L^{p_0'}(A_k)} + ||h(u)||_{L^{p_0}(\pa \Omega)} ||(u-k)^+||_{W^{1,p_0'}(\Omega)}, \nonumber \\
        & \leq ||f(u)||_{L^{p_0}(\Omega)} ||u-k||_{L^{p_0'}(A_k)} + ||h(u)||_{L^{p_0}(\pa \Omega)} ||\nabla u||_{L^{p_0'}(A_k)}, \nonumber
    \end{align}
    where $p_0' = \frac{p_0}{p_0-1}< N$. 
    
    Now we choose $r = N$, it follows from  H\"{o}lder inequality, Sobolev inequality and \eqref{Poincare}, we get
    \begin{align}\label{uk-4}
        ||u-k||_{L^{p_0'}(A_k)} & \leq ||(u-k)^+||_{L^N(A_k)}|A_k|^{\theta} \nonumber \\
        & \leq C(\Omega)|| (u - k)^+||_{W^{1,N}(A_k)}|A_k|^{\theta} \\
        & \leq C(\Omega) ||\nabla u||_{L^{N}(A_k)}|A_k|^{\theta}, \nonumber
    \end{align}
    where $\theta = \frac{1}{p_0'} - \frac{1}{N} $. 
    
    Besides, using H\"{o}lder inequality, we have
    \begin{align}\label{uk-5}
        ||\nabla u||_{L^{p_0'}(A_k)} \leq ||\nabla u||_{L^{N}(A_k)}|A_k|^{\theta}.
    \end{align}
    Inserting \eqref{uk-4} and \eqref{uk-5} into \eqref{u-k3} and applying Young inequality, we obtain
    \begin{align*}
        c_0 \int_{A_k} & |\nabla u|^N \ \dx  \leq \left( C( \Omega) ||f||_{L^{p_0}(\Omega)} |A_k|^{\theta} + ||h||_{L^{p_0}(\pa \Omega)} |A_k|^{\theta} \right) \left( \int_{A_k} |\nabla u|^N \ \dx \right)^{\frac{1}{N}} \\
        & \leq \frac{c_0}{2}\int_{A_k} |\nabla u|^N \ \dx + \left( C( \Omega, c_0) ||f||_{L^{p_0}(\Omega)}^{\frac{N}{N-1}} + C(c_0)||h||_{L^{p_0}(\pa \Omega)}^{\frac{N}{N-1}} \right)|A_k|^{\frac{N\theta}{N-1}} ,
    \end{align*}
    which implies
    \begin{align}\label{Ln}
        \int_{A_k} |\nabla u|^N \ \dx \leq C(\Omega, c_0) (||f||_{L^{p_0}(\Omega)}^{\frac{N}{N-1}} + ||h||_{L^{p_0}(\pa \Omega)}^{\frac{N}{N-1}} ) |A_k|^{\frac{N \theta}{N-1}}.
    \end{align}
    Next, for $l > k$, we  choose $s > N$, by using the similar arguments in \eqref{uk-4}, we have
    \begin{align*}
        |A_l| & \leq \int_{A_l} \left( \frac{(u-k)^+}{l-k} \right)^N \ \dx\\
        & \leq \frac{1}{(l-k)^N} \int_{A_k} [(u-k)^+]^N \ \dx \\
        & \leq \frac{1}{(l-k)^N} \left(\int_{A_k} [(u-k)^+]^s \ \dx \right)^{\frac{N}{s}} |A_k|^{1-\frac{N}{s}} \\
        & \leq \frac{C(\Omega)}{(l-k)^N} ||\nabla u||_{L^N({A_k})}^N |A_k|^{1 - \frac{N}{s}}\\
        & \leq \frac{C(\Omega, c_0)}{(l-k)^N} (||f||_{L^{p_0}(\Omega)}^{\frac{N}{N-1}} + ||h||_{L^{p_0}(\pa \Omega)}^{\frac{N}{N-1}} )|A_k|^{1 + \frac{N \theta}{N-1} - \frac{N}{s}} \\
        & \leq \frac{C(\Omega, c_0, q, N , ||h||_{L^{p_0}(\pa\Omega)}, || u||_{W^{1,N}(\Omega)})}{(l-k)^N} |A_k|^{1+ \delta},
    \end{align*}
    where $\delta = \frac{N\theta}{N-1} - \frac{N}{s}>0$ if we choose $s= s(p_0,N)$ large enough. Hence, Lemma \ref{iteration} tells us that there is a constant $M$ depending on $|| u||_{W^{1,N}(\Omega)}, ||h||_{L^{p_0}(\pa\Omega)},  N, \Omega, H$, such that $|A_M| = 0$, which means that
    \begin{align}\label{u+}
        \sup_{\Omega} u^+  \leq M(\Omega, H, N, ||h||_{L^{p_0}(\pa\Omega)},||\nabla u||_{L^N(\Omega)}).
    \end{align}
    Similarly, we can use the similar arguments to deal with $-u$ and obtain
    \begin{align}\label{u-}
        \sup_{\Omega} u^-  \leq \tilde{M} (\Omega, H, N, ||h||_{L^{p_0}(\pa\Omega)},||\nabla u||_{L^N(\Omega)}),
    \end{align}
    for a positive constant $\tilde{M}$.

     Next, we assume the condition (2) holds. In this case we only need to prove that $u^+$ is bounded. From \eqref{u-k}, we know that $h(u)(u-k)^+ \geq 0$, hence \eqref{u-k2} becomes
     \begin{align*}
         c_0 \int_{A_k} |\nabla u|^N \ \dx  \leq ||f(u)||_{L^{p_0}(\Omega)} ||u-k||_{L^{p_0'}(A_k)}.
     \end{align*}
     Using the similar arguments as in the case $(1)$, we obtain \eqref{u+}. This completes the proof.
\end{proof}

\begin{rmk}
    If $h$ satisfies $|h(t)| \leq C\e^{t}$ for some constant $C>0$. Then, from the Trudinger-Moser trace inequality (see \cite{C2008}), we have
    \begin{align*}
        \int_{\partial \Omega} \exp\left(\frac{\alpha |u(x) - u_{\pa \Omega}|^{\frac{N}{N-1}}}{||\nabla u||_{L^N(\Omega)}^{\frac{N}{N-1}}} \right)  \ \text{d} \sigma \leq C(\Omega), \ \ u_{\partial \Omega} = \frac{1}{|\partial \Omega|} \int_{\partial \Omega} u(x) \ \text{d} \sigma.
    \end{align*}
    From the Sobolev trace inequality, we have
    \begin{align*}
        |u_{\pa \Omega}| \leq C|\pa \Omega|^{-1} ||u||_{W^{1,N}(\Omega)}.
    \end{align*}
    Thus, we can use the same argument as \eqref{f-Lq} to prove that $h \in L^p(\partial \Omega)$ for any $1\leq p < +\infty$.
\end{rmk}

Next, we prove the  following Proposition for higher regularity. 
\begin{prop}\label{prop2-3}
    Suppose $\Omega \subset \R^N$ is a bounded convex domain, $f , h \in C^1(\R)$
    and $a$ satisfies
    \begin{align}\label{a-condition}
        \begin{cases}
            |a(\xi)| + |\pa a(\xi)|(\xi^2 + s^2)^{\frac{1}{2}} \leq L(|\xi|^2 + s^2)^{\frac{N-1}{2}}, \\
            \mu (|\xi|^2 + s^2)^{\frac{N-2}{2}} |z|^2 \leq \pa a(\xi)z \cdot z,
        \end{cases}
    \end{align}
    where $0 < \mu \leq L$ and $0 \leq s < 1$ are three constants. If $u \in W^{1,N}(\Omega)$ is a bounded weak solution of
    \begin{align*}
        \begin{cases}
            \text{div} (a(\nabla u )) + f(u) = 0, \ \ & \text{in} \ \ \Omega, \\
            a(\nabla u) \cdot \nu + h(u) = 0, \ \ & \text{on} \ \ \pa \Omega.
        \end{cases}
    \end{align*}
    Then, there is a constant $C = C(N ,L ,\mu, \Omega, ||u||_{C^1(\Omega)}, f, h)$, such that
    \begin{align*}
        ||a(\nabla u)||_{W^{1,2}(\Omega)} \leq C.
    \end{align*}
\end{prop}

\begin{proof}
    We follow the arguments of Proposition 2.2 in \cite{CCR2021}, see also Proposition 2.8 in \cite{CFR2020}. Choose a sequence of open convex smooth subdomain $\{\Omega_k\}_k$ of $\Omega$, such that $\lim_k \Omega_k = \Omega$, and suppose there is $\bar{x} \in \Omega_k$ for any $k$. We consider the following problem:
    \begin{align}\label{uk}
        \begin{cases}
            \text{div} (a(\nabla u_k )) + f(u) = 0, \ \ & \text{in} \ \ \Omega_k, \\
            u_k(\bar{x}) = u(\bar{x}), \\
            a(\nabla u_k) \cdot \nu + h(u) = 0, \ \ & \text{on} \ \ \pa \Omega_k.
        \end{cases}
    \end{align}
  In order to find the solution of \eqref{uk}, we first study the following minimization problem:
    \begin{align*}
        \min\limits_{v \in W^{1,N}(\Omega_k)} \left\{ \int_{\Omega_k} \left( \frac{1}{N}H^N(\nabla v) - f(u)v \right) \ \dx + \int_{\pa \Omega_k} h(u) v \ \text{d} \sigma \right\}.
    \end{align*}
    Since $\xi\mapsto H(\xi)^N/N$ is strictly convex and the remaining terms are linear in $v$, the functional is strictly convex modulo additive constants. Restricting the functional to the zero-mean subspace of $W^{1,N}(\Omega_k)$, the Poincaré inequality yields coercivity, which implies weak lower semicontinuity. Hence, the direct method gives a unique minimizer $v_k$ in the zero-mean class. Next, set $u_k(x) = v_k(x) + u(\bar{x}) - v_k(\bar{x})$, we find the unique solution $u_k$ of equation \eqref{uk}, which is uniformly bounded in $C^{1,\alpha}_{loc}(\Omega)$. Then, we choose a sequence of radially symmetric smooth mollifier and let
    \begin{align*}
        a^l(z) = (a * \phi_l)(z), \ \ z \in \R^N.
    \end{align*}
     Since ${\phi_l}$ is an approximation of the identity and $a$ is continuous, we have $a^l \to a$ locally uniformly in $\mathbb R^N$. Moreover, by the positivity and unit-mass property of mollifier $\phi_l$, condition \eqref{a-condition} is preserved under convolution, with $s$ replaced by $s_l$ with $s_l  \to 0$ as $l \to +\infty$.
      
      Let $u_{k,l}$ be the unique solution of the following equation:
    \begin{align}\label{ukl}
        \begin{cases}
            \text{div} (a^l(\nabla u_{k,l})) + f(u) = 0, \ \ & \text{in} \ \ \Omega_k,  \\
            u_{k,l}(\bar{x}) = u(\bar{x}), \\
            a^l(\nabla u_{k,l}) \cdot \nu + h(u) =0, \ \ &\text{on} \ \ \pa \Omega_k.
        \end{cases}
    \end{align}
    Since $u$ is bounded, we know that $f(u)$ and $h(u)$ are also bounded. From the standard elliptic regularity, we obtain that $\{u_{k,l}\}$ is bounded in $C^{1,\alpha}(\bar{\Omega}_k)$ uniformly in $l$. Let $ l \to +\infty$, we find that the limit function $\bar{u}_k$ satisfies
    \begin{align*}
        \begin{cases}
            \text{div} (a(\nabla \bar{u}_k )) + f(u) = 0, \ \ & \text{in} \ \ \Omega_k, \\
            \bar{u}_k(\bar{x}) = u(\bar{x}), \\
            a(\nabla \bar{u}_k) \cdot \nu + h(u) = 0, \ \ & \text{on} \ \ \pa \Omega_k.
        \end{cases}
    \end{align*}
    From the uniqueness of solution for \eqref{uk}, we find $\bar{u}_k = u_k$ and thus $u_{k,l} \to u_k$ in $C_{loc}^{1}(\Omega)$ as $l \to +\infty$. Similarly, we have $u_k \to u$ as $k \to \infty$.

    Test \eqref{ukl} with $ \psi \in C^{\infty}(\Omega_k)$, we have
    \begin{align}\label{psi-test}
        \int_{\Omega_k} a^l(\nabla u_{k,l}) \cdot \nabla \psi \ \dx & = -  \int_{\pa \Omega_k} a^l(\nabla u_{k,l}) \cdot \nu \psi \ \text{d} \sigma + \int_{\Omega_k} f(u) \psi \ \dx \\
         & = \int_{\pa \Omega_k} h(u) \psi \ \text{d} \sigma + \int_{\Omega_k} f(u) \psi \ \dx.  \nonumber
    \end{align}
    Define
    \begin{align*}
        \Omega_{k, \delta} = \{ x \in \Omega_k: \text{dist}(x, \pa \Omega_k) > \delta\}.
    \end{align*}
    Choose $\delta$ small enough, and $x \in \Omega_{k, \delta} \backslash \Omega_{k, 2\delta}$, we can find a point $y = y(x) \in \pa \Omega_{k,\delta}$, such that
    \begin{align*}
        x = y - |x-y| \nu(y),
    \end{align*}
    where $\nu(y)$ is the outer normal vector at $y$ on $\pa \Omega_{k, \delta}$. Choose $ \varphi \in C_c^{\infty}(\Omega)$, then $(\Omega_{k,\delta} \backslash \Omega_{k,2 \delta}) \cap \text{supp} \ \varphi$ is smooth, and can be parametrized on $\pa \Omega_{k, \delta}$ by a $C^1$ function $g$.

    Choose a cut-off function $\zeta_{\delta}$, which satisfies $ \zeta_{\delta} = 1$ in $\Omega_{k, 2\delta}$, $\zeta_{\delta} = 0$ in $\Omega_k \backslash \Omega_{k, \delta}$ with
    \begin{align*}
        \nabla \zeta_{\delta}(x) = - \frac{1}{\delta} \nu(y(x)), \ \ x \in \pa \Omega_{k, \delta}.
    \end{align*}
   We also choose $\psi  = \pa_m (\varphi \zeta_{\delta})$ for $m \in \{1,2,\cdots,N\}$ in \eqref{psi-test} and use integration by parts, we obtain
    \begin{align}\label{test1}
        \sum\limits_{i=1}^N \int_{\Omega_k} \pa_m a_i^l(\nabla u_{k,l}) \left( \zeta_{\delta}\pa_i \varphi + \varphi \pa_i \zeta_{\delta}\right) \ \dx = \int_{\Omega_k} \pa_m (f(u)) \varphi \zeta_{\delta}dx.
    \end{align}
    Write $w_i = \pa_m a_i^l(\nabla u_{k,l}) \varphi$, we have
    \begin{align*}
        &\qquad\int_{\Omega_k} w_i \pa_i \zeta_{\delta} \ \dx\\
         & = - \frac{1}{\delta}\int_{\Omega_{k, \delta} \backslash \Omega_{k, 2\delta}} \nu_i(y(x)) w_i \ \text{d}x \\
        & = - \frac{1}{\delta} \int_{\delta}^{2\delta} \text{d}t \int_{\pa \Omega_{k, \delta}} \nu_i(y(x)) w_i(y - t\nu(y)) |\text{det} (Dg)| \ \text{d} \sigma \\
        & = - \int_{1}^2 \text{d}s \int_{\pa \Omega_{k, s\delta}} \nu_i(y) w_i(y - s \delta\nu(y)) |\text{det} (Dg)| \ \text{d} \sigma .
    \end{align*}
    Note that $w_i$ is continuous, let $\delta \to 0$, we  obtain
    \begin{align*}
        \lim\limits_{\delta \to 0} \int_{\Omega_k} \pa_m a_i^l(\nabla u_{k,l}) \varphi \pa_i \zeta_{\delta} \ \dx = -\int_{\pa \Omega_k} \nu_i\varphi  \pa_m a_i^l(\nabla u_{k,l})  \ \text{d} \sigma.
    \end{align*}
    Hence, taking $\delta \to 0$ in \eqref{test1}, we obtain
    \begin{align}\label{test2}
        \sum\limits_{i=1}^N \bigg( \int_{\Omega_k}\pa_m a_i^l(\nabla u_{k,l}) \pa_i \varphi \ \dx - \int_{\pa \Omega_k} \pa_m a_i^l(\nabla u_{k,l})\varphi \nu_i \ \text{d} \sigma\bigg)  
         = \int_{\Omega_k} \pa_m (f(u)) \varphi \ \dx. 
    \end{align}
    Next, we consider the following set:
    \begin{align*}
        \Omega_{k,\delta}^{t} : = \{x \in \Omega_{k,\delta}: \text{dist}(x,\pa \Omega_{k,\delta}) > t\}.
    \end{align*}
    Then, if $x \in (\Omega_{k,\delta} \backslash \Omega_{k, 2\delta}) \cap \text{supp} \ \varphi $, then $x$ can be written as $x = y(x)-t\nu(y)$, then we have $x \in \pa \Omega_{k,\delta}^{t}$ and the outer normal vector $\nu(x)$ at $x$ coincides with $\nu(y)$. Thus, we can choose $$\varphi(x) = a_m^l(\nabla u_{k,l}(x)).$$ 
    Then
    \begin{align*}
        & \quad \sum_{i,m=1}^{N} \pa_m a_i^l(\nabla u_{k,l}(x))\varphi(x) \nu_i(x) \\
        &= \sum_{i,m=1}^{N} \left(\varphi(x) \pa_m  (a^l_i(\nabla u_{k,l}(x)) \cdot \nu_i(x)) - \varphi(x) a_i^l(\nabla u_{k,l}(x)) \pa_m \nu_i(x) \right)\\
        & = \sum_{i,m=1}^{N} a_m^l(\nabla u_{k,l}(x)) \left( \pa_m  (a^l_i(\nabla u_{k,l}(x)) \cdot \nu_i(x))  - a_i^l(\nabla u_{k,l}(x)) \pa_m \nu_i(x) \right)\\
        & = \sum_{m=1}^{N} a_m^l(\nabla u_{k,l}(x)) \pa_m  (a^l(\nabla u_{k,l}(x)) \cdot \nu(x)) - \Pi_x^t(A^{k,l}(x), A^{k,l}(x)),
    \end{align*}
    where $\Pi_{x}^{t}$ is the second fundamental form at $x \in \pa \Omega_{k,\delta}^{t}$ and $A^{k,l}(x)$ is the tangential part of $a^l(\nabla u_{k,l}(x))$. Since $\Omega_k$ is convex, we know that $\Pi_{x}^t$ is non-negative. Hence, we have
    \begin{align*}
        & \quad \sum_{m=1}^{N} \pa_m a_i^l(\nabla u_{k,l}(x))\varphi(x) \nu_i(x)  \\
        & \leq \sum_{m=1}^{N} a_m^l(\nabla u_{k,l}(x)) \pa_m  (a^l(\nabla u_{k,l}(x)) \cdot \nu(x)).
    \end{align*}
    Integrating the above inequality on $\pa \Omega_k$ and using \eqref{ukl}, we obtain
    \begin{align}\label{test3}
        & \quad \sum_{i,m=1}^{N} \int_{\pa \Omega_k}\pa_m a_i^l(\nabla u_{k,l}(x))\varphi(x) \nu_i(x) \ \text{d} \sigma \nonumber \\
        & \leq \sum_{m=1}^{N} \int_{\pa \Omega_k} a_m^l(\nabla u_{k,l}(x)) \pa_m  (a^l(\nabla u_{k,l}(x)) \cdot \nu(x)) \ \text{d} \sigma \nonumber \\
        & = \int_{\pa \Omega_k} a^l(\nabla u_{k,l}(x)) \cdot  \nabla (a^l(\nabla u_{k,l}(x)) \cdot \nu(x)) \ \text{d} \sigma \\
        & = -  \int_{\pa \Omega_k} h'(u)a^l(\nabla u_{k,l}(x)) \cdot  \nabla u(x) \ \text{d} \sigma. \nonumber
    \end{align}
    Inserting \eqref{test3} into \eqref{test2}, we obtain
    \begin{align*}
       &\quad  \sum\limits_{i,m=1}^N  \int_{\Omega_k}\pa_m a_i^l(\nabla u_{k,l}) \pa_i a_m^l(\nabla u_{k,l}) \ \dx \\
        & \leq N \int_{\Omega_k} \nabla (f(u)) \cdot a^l(\nabla u_{k,l})  \ \dx  - \int_{\pa \Omega_k} h'(u)a^l(\nabla u_{k,l}(x)) \cdot  \nabla u(x) \ \text{d} \sigma.
    \end{align*}
    Using the similar arguments as Lemma 4.5 in \cite{AKM2018}, we obtain
    \begin{align}\label{test4}
        \int_{\Omega_k} |\nabla a^l(\nabla u_{k,l}) |^2 \ \dx & \leq \int_{\Omega_k} |a^l(\nabla u_{k,l}) |^2 \ \dx +  C \int_{\Omega_k} |\nabla (f(u)) | | a^l(\nabla u_{k,l}) | \ \dx \\
        & \qquad + C  \int_{\pa \Omega_k}| h'(u)|a^l(\nabla u_{k,l}(x)) \cdot  \nabla u(x) \ \text{d} \sigma. \nonumber 
    \end{align}
    For fixed $k$, we first let $l\to\infty$. By the local uniform convergence $a^l\to a$, the convergence $u_{k,l}\to u_k$, the terms on the right-hand side of \eqref{test4} converge by dominated convergence, while the left-hand side of \eqref{test4} passes to the limit by weak lower semicontinuity. We then let $k\to\infty$. Since $\Omega_k \to\Omega$ and $u_k\to u$, we infer from dominated convergence and smooth convex approximation $\Omega_k\to\Omega$ that
    \begin{align}\label{nabla-a-L2}
        \int_{\Omega} |\nabla a(\nabla u) |^2 \ \dx & \leq \int_{\Omega} |a(\nabla u) |^2 \ \dx +  C \int_{\Omega} |\nabla (f(u)) \cdot a(\nabla u)|  \ \dx \\
        & \qquad + C  \int_{\pa \Omega} |h'(u)|H^N(\nabla u) \ \text{d} \sigma.\nonumber 
    \end{align}
    Since the right-hand side of \eqref{nabla-a-L2} is bounded by a constant that depends on $\Omega, ||u||_{C^1(\Omega)}, f,h, L, \mu$, we obtain $a(\nabla u) \in W^{1,2}(\Omega)$. This completes the proof.
\end{proof}

Using the estimate above, we can define the integral of $S_2(W)$ and $S_{ij}^2(W)$ in Section \ref{sec:newton}. However, we still need the admissible assumption in Definition \ref{def:classicalRA} and \ref{def:weakRA} to define the boundary integral $\int_{\pa \Omega} B_{\Omega}[u] \ \text{d} \sigma$ in Section \ref{sec:newton}. 

\section{Newton defect and the equality profile}\label{sec:newton}

In this section, we will prove Proposition \ref{Prop2}. We will follow the idea in \cite{CCR2021}. Let
\[
       v=Ne^{-u/N},\qquad \widehat H(\xi)=H(-\xi),\qquad
       \widehat a(\xi)=\widehat H(\xi)^{N-1}\nabla\widehat H(\xi),
\]
From \eqref{Main}, we know that
\begin{align}\label{v-equation}
    \begin{cases}
        \Delta_N^{\hat{H}} v = \frac{(N-1)\hat{H}^N(\nabla v)}{v} + \hat{f}(v),  \ \ &\text{in} \ \ \Omega, \\
        \hat{a}(\nabla v) \cdot \nu - \hat{h}(v) = 0 , \ \ &\text{on} \ \ \partial \Omega,
    \end{cases}
\end{align}
where
\begin{align*}
    \hat{f}(v) = N^{1-N} v^{N-1} f\left(-N\log \frac{v}{N}\right), \\  \hat{h}(v) = N^{1-N} v^{N-1} h\left(-N\log \frac{v}{N}\right),
\end{align*}
and
\begin{align*}
    \Delta_N^{\hat{H}} v = \text{div}(\hat{a}(\nabla v)), \ \ \text{with}\ \ \hat{a}(\xi) = \hat{H}^{N -1}(\xi) \cdot \nabla \hat{H}(\xi).
\end{align*}
Let
\begin{align*}
    V(\xi) = \frac{\hat{H}(\xi)}{N} , \ \ W = \nabla(\nabla_{\xi} V(\nabla v)) = V_{\xi_i \xi_j}(\nabla v) v_{ij} = \nabla \hat{a}(\nabla v),
\end{align*}
then, it holds that $ w_{ij} = \partial_j \hat{a}^i(\nabla v) $. Since $ v = N \e^{-\frac{u}{N}}$, we have
\begin{align*}
    w_{ij} & = - \partial_j \left( \e^{-\frac{(N-1)u}{N}} a^i(\nabla u) \right) \\
    & =- \left( \e^{-\frac{(N-1)u}{N}} \pa_ja^i(\nabla u)  - \frac{N-1}{N} \e^{-\frac{(N-1)u}{N}} \pa_j u \cdot a^i(\nabla u)\right).
\end{align*}
Define
\begin{align*}
    S^2_{ij}(W) = - w_{ji} + \delta_{ij} \text{tr}(W),
\end{align*}
and
\begin{align*}
    S_2(W) = \frac{1}{2} \sum_{i,j =1}^N S^2_{ij}(W) w_{ij} = \frac{1}{2}(-\text{tr}(W^2) + \text{tr}(W)^2).
\end{align*}
Then, we have the following  inequality:
\begin{align}\label{Newton}
    S_2(W) \leq \frac{N-1}{2N} \text{tr}(W)^2,
\end{align}
and the equality holds if and only if $W = \frac{\text{tr}(W)}{N} I$ (see \cite{CS2009}).

\begin{lemma}\label{lem:equality-profile}
Assume that the equality case in \eqref{Newton} holds.  Then either
\begin{align*}
    u(x) = N \log N - N \log (\ell \cdot x + c),
\end{align*}
with $\ell \in \R^N, c \in \R$ be chosen such that $\ell \cdot x + c > 0 $ in $\Omega$, or
\begin{align*}
    u(x) = N \log N - N \log (c_1 + c_2 H_0(x_0-x)^{\frac{N}{N-1}}),
\end{align*}
with $c_2 \neq 0$, $x_0 \in \R^N$ and $c_1, c_2$ be chosen such that $c_1 + c_2 H_0(x_0-x)^{\frac{N}{N-1}} $ in $\Omega$.
\end{lemma}

\begin{proof}
The equality case holds if and only if there is $\lambda(x)$ defined in $\Omega$, such that
\[
       W=\nabla(\widehat a(\nabla v))=\lambda(x)I, a.e. \hbox{ in }\Omega.
\]
Therefore, we have $\partial_j \widehat{a}_i(\nabla v)=0$ for $i\ne j$, $\partial_i \widehat{a}_i(\nabla v)=\lambda$, and for $i\ne j$,
\[
    \partial_j\lambda =\partial_j \partial_i \widehat{a}_i(\nabla v)=\partial_i\partial_j \widehat{a}_i(\nabla v)=0.
\]
Since $N\ge2$, it follows that $\nabla\lambda=0$ and therefore $\lambda$ is constant. Hence
\[
       \widehat{a}(\nabla v)=\lambda x+b, \ \ b \in \R^N.
\]
If $\lambda=0$, then $\widehat{a}(\nabla v)$ is constant vector. Since the strict convexity of $H^N$ implies the strict monotonicity of $\widehat a=\nabla(\widehat H^N/N)$, we obtain that $\nabla v$ is constant. Therefore, there are $c \in \R, \  \ell \in \R^N$, such that
\[
    v(x) = \ell \cdot x + c.
\]
Besides, from the definition of $v$, we must have
\begin{align*}
    v(x) = \ell \cdot x + c > 0, \ \ \text{in} \ \ \Omega.
\end{align*}
If $\lambda \ne 0$, we need to use the Legendre dual of $V(\xi) = \widehat H(\xi)^N / N$, namely
\[
       V^*(\xi)=\frac{N-1}{N} H_0(-\xi)^{\frac{N}{N-1}}.
\]
Since $\widehat{a}(\nabla v)=\nabla_{\xi} V(\nabla v) = \lambda x + b$, we have $\nabla v=\nabla V^*(\lambda x+b)$.  Integrating gives
\[
       v(x) = c_1 + c_2 H_0(x_0-x)^{\frac{N}{N-1}},
\]
with $c_2 \neq 0$, $x_0 \in \R^N$ and $c_1, c_2$ be chosen such that $c_1 + c_2 H_0(x_0-x)^{\frac{N}{N-1}} $ in $\Omega$.  since $v = N \e^{-\frac{u}{N}}$, the desired result holds.
\end{proof}

The following Lemma is from  \cite{BC2018} (Lemma 4.1):

\begin{lemma}\label{lem1}
    Suppose $W = \nabla (\nabla_{\xi} V(\nabla v)) = V_{\xi_i \xi_j}(\nabla v) v_{ij}$, where $v$ and $V$ are two functions in $\R^N$, then for any given $\gamma \in \R$, we have
    \begin{align*}
        2 v^{\gamma} S^2(W) =  \textup{div} (v^{\gamma} S^2_{ij}(W) V_{\xi_i}(\nabla v)) - \gamma v^{\gamma -1} S^2_{ij}(W) V_{\xi_i}(\nabla v) v_j.
    \end{align*}
    Moreover, if $V(\xi) = \frac{H^p(\xi)}{p}$ for some positive constant $p$, then it holds that
    \begin{align*}
        &\quad2 v^{1-N} S^2(W)  +  Np(N-1)(p-1) v^{-N-1} V^2(\nabla v) + (1-N)(2p-1)v^{-N} V(\nabla v) \Delta_p^{H} v \\
        & = \text{div} (v^{1-N} S^2_{ij}(W) V_{\xi_i}(\nabla v) + (1-N)(p-1) v^{-N}V(\nabla v) \nabla_{\xi} V(\nabla v) ),
    \end{align*}
    where $\Delta_p^H v = \textup{div}(H^{p-1}(\nabla v) \cdot \nabla H(\nabla v))$.
\end{lemma}

\begin{lemma}\label{lem2}
    For any $\varphi \in C^1_c(\Omega)$, it holds that
    \begin{align}
        & \quad \int_{\Omega} \bigg( 2 v^{1-N} S^2(W) +  N^2(N-1)^2 v^{-N-1} V^2(\nabla v) \nonumber \\
        & \hspace{9em }+ (1-N)(2N-1)v^{-N} V(\nabla v) \Delta_N^{\hat{H}} v \bigg) \varphi \ \dx \label{2-1} \\
        & = - \int_{\Omega} \varphi_j (v^{1-N} S^2_{ij}(W) V_{\xi_i}(\nabla v) - (N-1)^2 v^{-N} V(\nabla v) V_{\xi_j}(\nabla v) ) \ \dx. \nonumber
    \end{align}
\end{lemma}

This Lemma can be proved by using Lemma \ref{lem1} and an approximation argument similar to Lemma 3.3 in \cite{CFR2020}, see also Lemma 3.1 in \cite{CCR2021}. Here we omit the proof.

For a positive constant $t$, define
\begin{align*}
    \Omega_t = \{ x \in \Omega : \text{dist}(x,\partial \Omega) > t \}.
\end{align*}
Choose a small constant $\delta>0$ , define a Lipschitz function $\zeta_{\delta}: \R^N \to [0,1]$ such that
\begin{align*}
    \zeta_{\delta} = \begin{cases}
        1, \ \ \text{in} \ \ \Omega_{2 \delta}, \\
        0, \ \ \text{in} \ \ \Omega \backslash \Omega_{\delta},
    \end{cases}
\end{align*}
and
\begin{align*}
    \nabla \zeta_{\delta}(x) = - \frac{1}{\delta} \nu(y(x)), \ \ \text{in} \ \ \Omega_{\delta} \backslash \Omega_{2\delta},
\end{align*}
where $y(x) \in \partial \Omega_{\delta}$ is chosen such that
\begin{align*}
    \text{dist}(x, y(x)) = \text{dist}(x, \partial \Omega_{\delta}),
\end{align*}
and
\begin{align*}
     x = y(x) - \nu(y(x))|y(x)-x|.
\end{align*}
Let $\varphi = \zeta_{\delta}$ in Lemma \ref{lem2}, we have
\begin{lemma}\label{lem3}
    Suppose the condition in Lemma \ref{lem2} holds, then we have
    \begin{align}\label{3-1}
        & \quad \int_{\Omega} \bigg( 2 v S^2(W) +  N^2(N-1)^2 v^{-1} V^2(\nabla v) \nonumber \\
        & \hspace{9em} + (1-N)(2N -1) V(\nabla v) \Delta_N^{\hat{H}} v \bigg) v^{-N}\zeta_{\delta} \ \dx \nonumber\\
        & = \frac{1}{\delta} \int_{\Omega_{\delta} \backslash \Omega_{2 \delta}} \bigg( - \nabla(\hat{a}(\nabla v) \cdot \nu) \cdot \hat{a}(\nabla v) + \Pi _x(\hat{a}^T(\nabla v), \hat{a}^T(\nabla v)) \\
        & \hspace{9em} + \hat{a}(\nabla v)\cdot \nu \bigg( \frac{N-1}{N} \frac{\hat{H}^{N}(\nabla v)}{v} + \hat{f}(v)\bigg)\bigg) v^{1-N} \ \dx, \nonumber
    \end{align}
    where $\Pi_x(\cdot, \cdot)$ is the second fundamental form at $x$.
\end{lemma}

\begin{proof}
    Insert $\varphi = \zeta_{\delta}$ into \eqref{2-1}, we get
    \begin{align*}
        & \quad \int_{\Omega} \bigg(  2 v S^2(W) +  N^2(N-1)^2 v^{-1} V^2(\nabla v) + (1-N)(2N-1) V(\nabla v) \Delta_N^{\hat{H}} v \bigg) v^{-N}\zeta_{\delta} \ \dx \\
        &= \frac{1}{\delta} \int_{\Omega_{\delta} \backslash    \Omega_{2 \delta}} \nu_j (v^{1-N} S^2_{ij}(W) V_{\xi_i}(\nabla v) - (N-1)^2 v^{-N} V(\nabla v) V_{\xi_j}(\nabla v) ) \ \dx.
    \end{align*}
    Let
    \begin{align*}
        \Theta = v^{1-N} S^2_{ij}(W) V_{\xi_i}(\nabla v) \cdot \nu
        - (N-1)^2 v^{-N} V(\nabla v) \nabla V(\nabla v) \cdot \nu.
    \end{align*}
    Since
    \begin{align*}
        S^2_{ij}(W) = - w_{ji} + \delta_{ij} \text{Tr}(W) = - w_{ji} + \delta_{ij} \Delta_N^{\hat{H}} v,
    \end{align*}
    we have
    \begin{align*}
        S^2_{ij}(W) V_{\xi_i}(\nabla v) \cdot \nu & = - w_{ji} V_{\xi_i}(\nabla v) \cdot \nu_j + \Delta_N^{\hat{H}} v \cdot \nabla_{\xi} V(\nabla v) \cdot \nu \\
        & = - \pa_i \hat{a}_j(\nabla v) \cdot \hat{a}_i(\nabla v) \cdot \nu_j + \Delta_N^{\hat{H}} v \cdot \hat{a}(\nabla v) \cdot \nu \\
        & = - \pa_i (\hat{a} (\nabla v) \cdot \nu) \hat{a}_i(\nabla v) + \hat{a}_i(\nabla v) \hat{a}_j(\nabla v) \pa_i \nu_j + \Delta_N^{\hat{H}} v \cdot \hat{a}(\nabla v) \cdot \nu \\
        & = - \nabla(\hat{a}(\nabla v) \cdot \nu) \cdot \hat{a}(\nabla v) + \Pi_x(\hat{a}^T(\nabla v), \hat{a}^T(\nabla v)) \\
        &\quad + \Delta_N^{\hat{H}} v \cdot \hat{a}(\nabla v) \cdot \nu  \\
        & = -\nabla(\hat{a}(\nabla v) \cdot \nu) \cdot \hat{a}(\nabla v) + (N-1)\hat{a}(\nabla v) \cdot \nu  \frac{\hat{H}^N(\nabla v)}{v} \\
        & \quad+ \hat{a}(\nabla v) \cdot \nu \hat{f}(v) + \Pi_x(\hat{a}^T(\nabla v), \hat{a}^T(\nabla v)),
    \end{align*}
    and
    \begin{align*}
        v^{-N} V(\nabla v) \nabla V(\nabla v) \cdot \nu & = v^{-N} \frac{\hat{H}^N(\nabla v)}{N} \hat{H}^{N-1}(\nabla v) \nabla \hat{H}(\nabla v) \cdot \nu \\
        & = v^{-N} \frac{\hat{H}^N(\nabla v)}{N} \hat{a}(\nabla v) \cdot  \nu.
    \end{align*}
    Therefore, we have
    \begin{align*}
        \Theta = &\bigg( - \nabla(\hat{a}(\nabla v) \cdot \nu) \cdot \hat{a}(\nabla v) + \Pi _x(\hat{a}^T(\nabla v), \hat{a}^T(\nabla v)) \\
        & \quad + \hat{a}(\nabla v)\cdot \nu \bigg( \frac{N-1}{N} \frac{\hat{H}(\nabla v)}{v} + \hat{f}(v)\bigg)\bigg) v^{1-N}.
    \end{align*}
    Thus \eqref{3-1} holds.
\end{proof}

\begin{proof}[Proof of Proposition \ref{Prop2}]
    Write $c_N = N^{1-N}$. Using the Newton type inequality, we get
    \begin{align*}
        & \text{L.H.S of \eqref{3-1}}  \leq \int_{\Omega} \bigg( \frac{N-1}{N} v^{1-N} (\Delta_N^{\hat{H}} v)^2 +  N^2(N-1)^2 v^{-N-1} V^2(\nabla v) \\
        &\qquad + (1-N)(2N-1)v^{-N} V(\nabla v) \Delta_N^{\hat{H}} v \bigg) \zeta_{\delta} \ \dx\\
        & = \int_{\Omega} \bigg( \frac{N-1}{N} v^{1-N} \bigg( \frac{(N-1)\hat{H}^N(\nabla v)}{v} + \hat{f}(v) \bigg)^2 +  N^2(N-1)^2 v^{-N-1} V^2(\nabla v) \\
        & \qquad + (1- N)(2N-1)v^{-N} V(\nabla v) \bigg(\frac{(N-1)\hat{H}^N(\nabla v)}{v} + \hat{f}(v)\bigg) \bigg) \zeta_{\delta} \ dx\\
        & = \frac{N-1}{N} \int_{\Omega} v^{1-N} \hat{f}^2(v) \zeta_{\delta} \ \dx - \frac{N-1}{N} \int_{\Omega}v^{-N} \hat{H}^{N}(\nabla v) \hat{f}(v) \zeta_{\delta} \ \dx \\
        & = c_N^2 \frac{N-1}{N} \int_{\Omega} v^{N-1} f^2\bigg( - N \log \frac{v}{N} \bigg) \zeta_{\delta} \ \dx \\
        & \qquad - c_N \frac{N-1}{N} \int_{\Omega} v^{-1} \hat{H}^{H}(\nabla v) f\bigg( - N \log \frac{v}{N} \bigg) \zeta_{\delta} \ \dx.
    \end{align*}
    Thus, it follows that
    \begin{align*}
        &\quad c_N \frac{N-1}{N} \left( c_N\int_{\Omega} v^{N-1} f^2\bigg( - N \log \frac{v}{N} \bigg) \zeta_{\delta} \ \dx - \int_{\Omega} v^{-1} \hat{H}^{N}(\nabla v) f\bigg( - N \log \frac{v}{N} \bigg) \zeta_{\delta} \ \dx \right) \\
        & \geq \frac{1}{\delta} \int_{\Omega_{\delta} \backslash \Omega_{2 \delta}} \bigg( - \nabla(\hat{a}(\nabla v) \cdot \nu) \cdot \hat{a}(\nabla v) + \Pi _x(\hat{a}^T(\nabla v), \hat{a}^T(\nabla v)) \\
        & \qquad + \hat{a}(\nabla v)\cdot \nu \bigg( \frac{N-1}{N} \frac{\hat{H}(\nabla v)}{v} + \hat{f}(v)\bigg)\bigg) v^{1-N} \ \dx.
    \end{align*}
    Inserting $v = N \e^{-\frac{u}{N}}$ into the above inequality and divided by $c_N$, we obtain
    \begin{align}
        & \quad \frac{N-1}{N} \int_{\Omega} \e^{-\frac{(N-1)u}{N}} f^2(u) \zeta_{\delta} \ \dx - \frac{N-1}{N^2} \int_{\Omega} e^{-\frac{(N-1)u}{N}} H^{N}(\nabla u) f(u) \zeta_{\delta} \ \dx \nonumber \\
        & \geq \frac{(N-1)^2}{N^2 \delta} \int_{\Omega_{\delta} \backslash \Omega_{2\delta}} \e^{-\frac{(N-1)u}{N}} H^N(\nabla u) a(\nabla u) \cdot \nu \ dx \label{inequ-1} \\
        & \qquad -\frac{1}{\delta} \int_{\Omega_{\delta} \backslash \Omega_{2\delta}} \e^{-\frac{(N-1)u}{N}}  a(\nabla u) \cdot \nu f(u)  \ \dx \nonumber\\
        & \qquad -\frac{1}{\delta} \int_{\Omega_{\delta} \backslash \Omega_{2\delta}} \e^{-\frac{(N-1)u}{N}} \nabla (a(\nabla u) \cdot \nu) \cdot a(\nabla u) \ \dx \nonumber \\
        & \qquad  + \frac{1}{\delta} \int_{\Omega_{\delta} \backslash \Omega_{2\delta}} \e^{-\frac{(N-1)u}{N}} \Pi_x( a^T(\nabla u),a^T(\nabla u) ) \ \dx. \nonumber
    \end{align}
    Next, test \eqref{Main} with $\e^{-\frac{(N-1)u}{N}} f(u) \zeta_{\delta}$, we get
    \begin{align*}
        & \quad \int_{\Omega} \e^{-\frac{(N-1)u}{N}} f^2(u) \zeta_{\delta} \ \dx +\frac{N-1}{N} \int_{\Omega} e^{-\frac{(N-1)u}{N}} H^{N}(\nabla u) f(u) \zeta_{\delta} \ \dx \\
        & = \int_{\Omega} \e^{-\frac{(N-1)u}{N}} H^N(\nabla u) f'(u) \zeta_{\delta} \ \dx - \frac{1}{\delta}\int_{\Omega_{\delta}\backslash \Omega_{2\delta}}\e^{-\frac{(N-1)u}{N}} f(u) a(\nabla u) \cdot \nu \ \dx.
     \end{align*}
     Recall $\Phi(u) = \e^{-u} f(u) $ and write
     \begin{align*}
         f'(u) = \e^{u} \Phi'(u) + f(u),
     \end{align*}
     we get
     \begin{align}
        & \quad \frac{N-1}{N}\int_{\Omega} \e^{-\frac{(N-1)u}{N}}  f^2(u) \zeta_{\delta} \ \dx  - \frac{N-1}{N^2} \int_{\Omega} e^{-\frac{(N-1)u}{N}} H^{N}(\nabla u) f(u) \zeta_{\delta} \ \dx  \label{equa-2}\\
        & = \frac{N-1}{N}\int_{\Omega} \e^{\frac{u}{N}} H^N(\nabla u) \Phi'(u) \zeta_{\delta} \ \dx - \frac{N-1}{N\delta}\int_{\Omega_{\delta}\backslash \Omega_{2\delta}}\e^{-\frac{(N-1)u}{N}} f(u) a(\nabla u) \cdot \nu \ \dx. \nonumber
    \end{align}
    Combine \eqref{inequ-1} and \eqref{equa-2}, we obtain
    \begin{align}\label{inequa-2}
       &\quad  \frac{N-1}{N}\int_{\Omega} \e^{\frac{u}{N}} H^N(\nabla u) \Phi'(u) \zeta_{\delta} \ \dx \\
        & \geq \frac{(N-1)^2}{N^2 \delta} \int_{\Omega_{\delta} \backslash \Omega_{2\delta}} \e^{-\frac{(N-1)u}{N}} H^N(\nabla u) a(\nabla u) \cdot \nu \ dx \nonumber \\
        &  -\frac{1}{\delta} \int_{\Omega_{\delta} \backslash \Omega_{2\delta}} \e^{-\frac{(N-1)u}{N}} \nabla (a(\nabla u) \cdot \nu) \cdot a(\nabla u) \ \dx  \nonumber\\
        &  -\frac{1}{N\delta} \int_{\Omega_{\delta} \backslash \Omega_{2\delta}} \e^{-\frac{(N-1)u}{N}}  a(\nabla u) \cdot \nu f(u)  \ \dx  \\
        &   + \frac{1}{\delta} \int_{\Omega_{\delta} \backslash \Omega_{2\delta}} \e^{-\frac{(N-1)u}{N}} \Pi_x( a^T(\nabla u),a^T(\nabla u) ) \ \dx  \nonumber \\
        & = \frac{1}{\delta} \int_{\Omega_{\delta} \backslash \Omega_{2\delta}} B_{\Omega}(u) \ \dx. \nonumber
    \end{align}
    Let $\delta \to 0$ in \eqref{inequa-2}, we obtain
    \begin{align*}
        \frac{N-1}{N} \int_{\Omega} \e^{\frac{u}{N}} H^N(\nabla u) \Phi'(u) \ \dx \geq \int_{\pa \Omega} B_{\Omega}[u] \ \text{d} \sigma.
    \end{align*}
    The equality holds if and only if $W(x) = \lambda(x) \text{Id}$ in $\Omega$. In this case, we deduce from Lemma \ref{lem:equality-profile} that either
    \begin{align*}
        u(x) = N \log N - N \log (\ell \cdot x + c),
    \end{align*}
    with $\ell \in \R^N, c \in \R$ chosen such that $\ell \cdot x + c > 0 $ in $\Omega$, or
    \begin{align*}
        u(x) = N \log N - N \log (c_1 + c_2 H_0(x_0-x)^{\frac{N}{N-1}}),
    \end{align*}
    with $c_2 \neq 0$, $x_0 \in \R^N$ and $c_1, c_2$ be chosen such that $c_1 + c_2 H_0(x_0-x)^{\frac{N}{N-1}} $ in $\Omega$.
\end{proof}

\section{Proof of Main Theorems in Bounded Domain}\label{sec:sufficient}

The key integral inequality established in Proposition \ref{Prop2} provides the main tool for proving the rigidity results stated in Section~1. We first prove Theorem~1.8. Under the monotonicity assumption on $\Phi$ and the boundary sign condition~\eqref{unbounded-assumption}, the inequality in Proposition~1.7 must become an equality. The corresponding equality characterization then yields the desired classification of solutions.
\begin{proof}[Proof of Theorem \ref{Thm3}]
By \eqref{Thm1-assumption},
\[
       0 \geq \frac{N-1}{N}\int_\Omega e^{u/N}H^N(\nabla u)\Phi'(u)dx \geq \int_{\partial\Omega}B_\Omega[u]d\sigma\ge0 .
\]
In particular, the equality  holds. Therefore, Lemma \ref{lem:equality-profile} gives the desired result.
\end{proof}

\begin{proof}[Proof of Theorem \ref{Thm1}]
For the pure Neumann problem $h\equiv0$, we have $a(\nabla u) \cdot \nu =0$ on $\pa \Omega$. In this case, the boundary integral $\int_{\pa \Omega} B_{\Omega}[u] \ \text{d} \sigma$ only involves the integration of $\Pi_x(a^T(\nabla u), a^T(\nabla u))$ and thus we do not need to assume $u$ is an admissible weak solution to use the argument of Proposition \ref{prop2-3} which yields
\[
       0 \geq \frac{N-1}{N} \int_{\Omega} \e^{\frac{u}{N}} H^N(\nabla u) \Phi'(u) \ \dx \geq \int_{\pa \Omega} \e^{-\frac{(N-1)u}{N}} \Pi_x(a^T(\nabla u), a^T(\nabla u)) \ \text{d} \sigma \geq 0.
\]
Hence the equality  holds and $u$ is of the form \eqref{affine} or \eqref{bubble-form-intro}.
 
 Suppose first that $v=c+\ell\cdot x$ with $\ell\ne0$. Then $\widehat{a}(\nabla v) = \widehat a(\ell)$ is a non-zero constant vector. 
 Multiplying \eqref{v-equation} by $\hat{a}(\nabla v) \cdot x$, integrating over $\partial\Omega$, and using the divergence theorem, we obtain
\begin{align*}
    0&=\int_{\partial\Omega}( \hat{a}(\nabla v)  \cdot x)(\hat{a}(\nabla v)  \cdot \nu)\,d\sigma\\
    &=\int_\Omega\operatorname{div}\bigl((\hat{a}(\nabla v) \cdot x) \hat{a}(\nabla v)  \bigr)\,dx =|\hat{a}(\ell) |^2|\Omega|,
\end{align*}
which is impossible. Therefore, if $u$ takes the form \eqref{affine}, then we have $\ell =0$ and $u$ is a constant.

Suppose next that $u$ has the form \eqref{bubble-form-intro}. Then, from the proof of \eqref{lem:equality-profile}, we have
\begin{align*}
       \hat{a}(\nabla v)  =\lambda (x - x_0),\ \ \lambda\ne0, \ \ x_0 \in \R^N.
\end{align*}
Then \eqref{v-equation} gives $(x-x_0)\cdot\nu=0$ on $\partial\Omega$.  Hence
\begin{align*}
0&=\int_{\partial\Omega}(x-x_0)\cdot\nu\,d\sigma
 =\int_\Omega\operatorname{div}(x-x_0)\,dx
 =N|\Omega|,
\end{align*}
again a contradiction. Hence, $u$ must be a constant.
\end{proof}

Next, we give some sufficient conditions for $\int_{\pa \Omega} B_{\Omega}[u] \ \text{d}\sigma \geq 0$. We first give an equivalent expression of $\int_{\pa \Omega} B_{\Omega} [u] \ \text{d} \sigma$.  We denote by
\[
    \mathcal{H}_{\pa \Omega}(x):=\frac{1}{N-1}\operatorname{div}_\Gamma\nu(x)
\]
the mean curvature of $\partial\Omega$ with respect to the
outward unit normal $\nu$. Thus $\mathcal{H}_{\pa \Omega} \geq0$ if $\Omega$ is convex.

For $x\in\partial\Omega$, $s\in\R$, and $\tau\in T_x\partial\Omega$, let $\rho=\rho(x,s,\tau)$ be the
unique real number satisfying
\begin{align}\label{rho-eq}
    a\bigl(\tau+\rho\nu(x)\bigr)\cdot\nu(x)=-h(s).
\end{align}
The uniqueness follows from the strict convexity and coercivity of
\[
    \rho\longmapsto\frac1N H^N\bigl(\tau+\rho\nu(x)\bigr).
\]
Set
\[
    \xi=\xi(x,s,\tau):=\tau+\rho(x,s,\tau)\nu(x).
\]

\begin{lemma}\label{lem:boundary-reduction}
Define
\begin{align}\label{R-density}
    \mathcal R_{f,h,\Omega,H}(x,s,\tau) = & \Pi_x\bigl(a(\xi)+h(s)\nu,a(\xi)+h(s)\nu\bigr) +(N-1) \mathcal{H}_{\pa \Omega}(x)h(s)^2 -\frac{N-1}{N} h(s)f(s) \nonumber\\
    &-\frac{(N-1)^2}{N^2}h(s)H^N(\xi) +\bigg(2h'(s)- \frac{N-1}{N} h(s)\bigg)a(\xi)\cdot\tau .
\end{align}
Then every classical solution of \eqref{Main} satisfies
\begin{align}\label{eq:int-reduction}
    \int_{\partial\Omega}B_\Omega[u]\,d\sigma = \int_{\partial\Omega}e^{- \frac{(N-1)u}{N}} \mathcal R_{f,h,\Omega,H} \bigl(x,u,\nabla_{\pa \Omega} u\bigr)\,d\sigma,
\end{align}
where $\nabla_{\pa \Omega}$ is the derivative along $\pa \Omega$. The same identity holds for admissible weak solutions by
approximation.
\end{lemma}

\begin{proof}
On $\partial\Omega$, the divergence decomposition gives
\begin{align}\label{boundary-div-decomposition}
    \operatorname{div}a(\nabla u) = \partial_\nu\bigl(a(\nabla u)\cdot\nu\bigr) + \operatorname{div}_{\pa \Omega} \bigl(a(\nabla u)+h(u)\nu\bigr) 
    +(N-1) \mathcal{H}_{\pa \Omega} \cdot \bigl(a(\nabla u)\cdot\nu\bigr).
\end{align}
Since
\[
    \operatorname{div}a(\nabla u)=-f(u), \qquad a(\nabla u)\cdot\nu=-h(u),
\]
we have
\begin{align}\label{normal-flux}
    \partial_\nu\bigl(a(\nabla u)\cdot\nu\bigr) = -f(u) -\operatorname{div}_{\pa \Omega} \bigl(a(\nabla u) + h(u)\nu\bigr) + (N-1)\mathcal{H}_{\pa \Omega} h(u).
\end{align}
Moreover, differentiating the Robin condition only in tangential
directions gives
\begin{align*}
    \nabla_{\pa \Omega}\bigl(a(\nabla u)\cdot\nu\bigr) =-h'(u)\nabla_{\pa \Omega} u.
\end{align*}
Hence,
\begin{align}\label{4-1-1}
    \nabla(a(\nabla u) \cdot \nu) \cdot a(\nabla u) & = \pa_{\nu} (a(\nabla u) \cdot \nu) \cdot a(\nabla u) \cdot \nu + \nabla_{\pa \Omega} (a(\nabla u) \cdot \nu) \cdot a(\nabla u) \\
    & = -h(u) \pa_{\nu} (a(\nabla u) \cdot \nu) - h'(u) \nabla_{\pa \Omega} u \cdot a(\nabla u) \nonumber
\end{align}
From \eqref{eq:BOmega}, \eqref{normal-flux} and \eqref{4-1-1}, we obtain
\begin{align}\label{B-before-IBP}
    B_\Omega[u] =e^{-\frac{N-1}{N} u}\Bigl[ & \Pi\bigl(a(\nabla u)+h(u)\nu, a(\nabla u)+h(u)\nu\bigr) +(N-1)\mathcal{H}_{\pa \Omega} h(u)^2 - \frac{N-1}{N} h(u)f(u) \nonumber\\
    & -\frac{(N-1)^2}{N^2}h(u)H^N(\nabla u) -h(u) \operatorname{div}_{\pa \Omega} \bigl(a(\nabla u)+h(u)\nu\bigr) \nonumber\\
    & +h'(u)a(\nabla u)\cdot\nabla_{\pa \Omega} u \Bigr].
\end{align}

Since $\partial\Omega$ is closed, we can use integration by parts on  surface to get
\begin{align}
    &\quad-\int_{\partial\Omega}e^{- \frac{(N-1)u}{N}}h(u) \operatorname{div}_{\pa \Omega} \bigl(a(\nabla u)+h(u)\nu\bigr)\,d\sigma \nonumber\\
    & = \int_{\partial\Omega}e^{- \frac{(N-1)u}{N}} \bigg( h'(u)- \frac{N-1}{N} h(u) \bigg) a(\nabla u)\cdot\nabla_{\pa \Omega} u\,d\sigma.
\end{align}
Combining this with \eqref{B-before-IBP} yields the coefficient $2h'(u)- \frac{N-1}{N} h(u)$.

Finally, taking
\[
    s=u(x),\qquad \tau=\nabla_{\pa \Omega} u(x),
\]
the Robin condition and the uniqueness in \eqref{rho-eq} imply
\[
    \xi\bigl(x,u,\nabla_{\pa \Omega} u\bigr)=\nabla u.
\]
Hence \eqref{eq:int-reduction} follows.
\end{proof}

\begin{prop}\label{cor:simple-RRC}
Assume that $\Omega$ is convex and that, for every $t\in\R$,
\begin{align}\label{simple-RRC}
       h(s)\le0,\qquad 2h'(s)=\frac{N-1}{N} h(s),
       \qquad (N-1)H_*h(s)^2-\frac{N-1}{N} h(s)f(s)\ge0,
\end{align}
where $H_* = \min_{x\in\partial\Omega}H_{\pa \Omega}(x)$ and $H_{\pa \Omega}(x)$ is the mean curvature at $x \in \pa \Omega$. Then $\int_{\pa \Omega} B_{\Omega}[u] \geq 0$.
\end{prop}

\begin{proof}
Under \eqref{simple-RRC}, the mixed term in \eqref{R-density} vanishes. Convexity gives $\Pi_x \ge 0$. Besides, since $$(N-1)H_{\pa \Omega}(x) h(s)^2-\frac{N-1}{N} h(s)f(s)\ge0$$ and $$-\frac{N-1}{2N}h(s)H^N(\xi)\ge0.$$ We deduce that every term in \eqref{R-density} is non-negative.
\end{proof}

Next, we consider the special case when $H(\xi)=|\xi|$. Then $a(\xi)=|\xi|^{N-2}\xi$, and $\rho$ is determined by
\begin{align}\label{rho-isotropic}
       (|\tau|^2+\rho^2)^{(N-2)/2}\rho=-h(s).
\end{align}
Write
\begin{align*}
       \Lambda=(|\tau|^2+\rho^2)^{1/2}.
\end{align*}
Then
\begin{align*}
       a(\nabla u) + h(u) \nu =\Lambda^{N-2}\tau,
       \qquad H^N (\xi)=\Lambda^N,
\end{align*}
and \eqref{R-density} becomes
\begin{align}\label{isotropic-RRC}
    \calR_{f,h,\Omega,|\cdot|}(x,s,\tau) = &\Lambda^{2N-4}\Pi_x(\tau,\tau) + \bigg(2h'(s)- \frac{(N-1)(2N-1)}{N^2} h(s)\bigg) \Lambda^{N-2} |\tau|^2\nonumber\\
    & + (N-1)H_{\pa \Omega}(x)h(s)^2- \frac{N-1}{N}h(s)f(s) - \frac{(N-1)^2}{N^2} h(s)\Lambda^{N-2}\rho^2.
\end{align}

We first consider a non-positive Robin term.
\begin{prop}\label{prop:isotropic-negative-h}
 Suppose that $H(\xi)=|\xi|$ and for every $s\in \R$,
\begin{align}\label{negative-h-conditions}
       h(s)&\le0,\nonumber\\
       2h'(s)- \frac{(N-1)(2N-1)}{N^2} h(s)&\ge0,\\
       H_*h(s)^2- \frac{N-1}{N} h(s)f(s)&\ge0.\nonumber
\end{align}
Then $$\int_{\pa \Omega} B_{\Omega}[u] \ \text{d} \sigma \geq 0.$$
\end{prop}

\begin{proof}
The first term in \eqref{isotropic-RRC} is non-negative by convexity.  The second term is non-negative by \eqref{negative-h-conditions}.  Since $h(s)\le0$, and
$$ -h(s)\Lambda^{N-2}\rho^2\ge0.$$ The remaining two scalar terms are bounded from below by $H_*h(s)^2- \frac{N-1}{N}h(s)f(s)$, which is non-negative.  Hence $$\int_{\pa \Omega} B_{\Omega}[u]\ \text{d} \sigma  \geq 0.$$
\end{proof}

Next, we give a condition for a non-negative Robin term.
\begin{prop}\label{prop:isotropic-positive-h}
Suppose that $H(\xi)=|\xi|$ and for every $s\in \R$,
\begin{align}\label{positive-h-conditions}
       h(s)&\ge0,\nonumber\\
       2h'(s)- \frac{(N-1)(2N-1)}{N^2} h(s)&\ge0,\\
       H_*h(s)^2- \frac{N-1}{N} h(s)f(s)-\frac{(N-1)^2}{N^2}h(s)^{\frac{2N-1}{N-1}}&\ge0.\nonumber
\end{align}
Then $$\int_{\pa \Omega} B_{\Omega}[u]  \ \text{d} \sigma \geq 0.$$
\end{prop}

\begin{proof}
From \eqref{rho-isotropic},
\begin{align*}
       h(s)=\Lambda^{N-2}|\rho|\ge|\rho|^{N-1}.
\end{align*}
Therefore,
\begin{align*}
       |\rho|\le h(s)^{1/(N-1)},
       \qquad \Lambda^{N-2}\rho^2=h(s)|\rho|\le h(s)^{\frac{N}{N-1}}.
\end{align*}
The first two terms in \eqref{isotropic-RRC} are non-negative, and the last three terms are bounded from below by the last line of \eqref{positive-h-conditions}.  This proves the proposition.
\end{proof}

\begin{proof}[Proof of Corollary \ref{cor4}]
Corollary \ref{cor4} follows directly from Proposition \ref{cor:simple-RRC}, Proposition \ref{prop:isotropic-negative-h} and Proposition \ref{prop:isotropic-positive-h}.
\end{proof}

\begin{rmk}\label{ex:affine-equality-profile}
We would like to point out that the form \eqref{linear-log} in Theorem \ref{Thm3} cannot be omitted. For example, let $b>1$ and set $\Omega = B_1(0) \subset \R^2$,
\begin{align*}
       U(x)=2\log2-2\log(b+x_1), \qquad f(s)=-\frac12e^s,
       \qquad h(s)=2-be^{s/2}
\end{align*}
Put $L=b+x_1$.  Then
\begin{align*}
       \nabla U=-\frac{2}{L} e_1, \qquad \Delta U=\frac2{L^2},  \qquad f(U)=-\frac2{L^2}.
\end{align*}
Thus $\Delta U+f(U)=0$. On $\partial B_1$,
\begin{align*}
       \partial_\nu U=-\frac{2x_1}{L},
       \qquad h(U)=2-\frac{2b}{L}=\frac{2x_1}{L},
\end{align*}
and therefore the Robin condition is satisfied.  Moreover,
\begin{align*}
       \Phi(s)=e^{-s}f(s)=-\frac12,
       \qquad 2e^{-U/2}=b+x_1.
\end{align*}

We finally verify the boundary condition \eqref{B-Omega}.  Parametrize $\partial B_1$ by $x_1=\cos\theta$ and write $L=b+\cos\theta$. Then, we have
\begin{align*}
       \calR_{f,h,B_1,|\cdot|}
       (x,U,\nabla_{\pa \Omega} U)
       =\frac{4(\cos^3\theta+2b\cos^2\theta-b)}{L^3}.
\end{align*}
Since $e^{-U/2}=L/2$,
\begin{align*}
e^{-U/2}\calR_{f,h,B_1,|\cdot|}
       (x,U,\nabla_{\pa \Omega} U)
       =2\frac{d}{d\theta}
       \left(\frac{\sin\theta\cos\theta}{b+\cos\theta}\right).
\end{align*}
Integration over $[0,2\pi]$ gives $\int_{\pa \Omega} B_{\Omega}[U] \ \text{d} \sigma =0$.
\end{rmk}

\begin{rmk}\label{ex:BR-is-essential}
The sign condition $\int_{\pa \Omega} B_{\Omega}[u] \ \text{d} \sigma \geq 0$ cannot be removed from Theorem \ref{Thm3}. For instance, let
\begin{align*}
       f(s)\equiv0,\qquad h(s)=-s,
       \qquad u_\varepsilon(x)=\varepsilon x_1
\end{align*}
in $B_1(0)\subset\R^2$, where $\varepsilon\ne0$.  Then
\begin{align*}
       \Delta u_\varepsilon=0,
       \qquad
       \partial_\nu u_\varepsilon+h(u_\varepsilon)
       =\varepsilon x_1-\varepsilon x_1=0
\end{align*}
on $\partial B_1$, and $\Phi'\equiv0$. However, $u_{\varepsilon}$ is not of the form \eqref{linear-log+} or \eqref{bubble-log+}.

To determine the sign of $\int_{\pa \Omega} B_{\Omega}[u] \ \text{d} \sigma$, write $x_1=\cos\theta=c$. Then,
\begin{align*}
    \calR_{0,-\operatorname{id},B_1,|\cdot|}
       (x,u_\varepsilon,\nabla_{\pa \Omega} u_\varepsilon) = &\varepsilon^2 (2c^2-1) + \varepsilon^3 \left(\frac{3}{4}c - \frac{1}{2}c^3\right).
\end{align*}
Expanding the weight $e^{-\varepsilon c/2}$ and using
\begin{align*}
       \int_0^{2\pi}c^2\,d\theta=\pi,
       \qquad
       \int_0^{2\pi}c^4\,d\theta=\frac{3\pi}{4},
\end{align*}
we obtain
\begin{align*}
      \int_{\pa \Omega} B_{\Omega}[u_{\varepsilon}] \ \text{d} \sigma
       =-\frac\pi8\varepsilon^4+O(\varepsilon^6)<0,
\end{align*}
for every sufficiently small $\varepsilon\ne0$.
\end{rmk}



Finally, we prove Theorem \ref{thm:unbounded}.
\begin{proof}[Proof of Theorem \ref{thm:unbounded}]
Applying the similar arguments as  in the proof of  Proposition \ref{prop2-3} on $D_j$, we obtain
\begin{align*}
    \frac{N-1}{N}\int_{D_j} e^{u/N}H^N(\nabla u)\Phi'(u)\,dx \ge & \int_{\pa D_j} \overline{B}_{\Omega}[u] \, \text{d} \sigma .
\end{align*}
It follows from the condition \eqref{unbounded-assumption} and $\Phi'(t) \leq 0$, we obtain
\[
    0 \ge \frac{N-1}{N}\int_{D_j} e^{u/N}H^N(\nabla u)\Phi'(u)\,dx  \geq \int_{\pa D_j} \overline{B}_{\Omega}[u] \, \text{d} \sigma \ge0.
\]
Hence equality holds almost everywhere, which implies that $u$ is of the form \eqref{linear-log} or \eqref{bubble-log}.
\end{proof}

\textbf{Acknowledgments} Yuxia Guo was supported by the National Natural Science Foundation of China (No. 12271283) and National Key R\&D Program (2023YFA1010002). Yichen Hu is supported by NNSF of China (No. 12501135) and the Fundamental Research Funds for the Central Universities (DUT24RC(3)110). Shaolong Peng was supported by the National Natural Science Foundation of China (Nos. 12401148 and 12571113), and Beijing Natural Science Foundation (No. 1262014).

\textbf{Statements and Declarations}
The authors confirm that there are no relevant financial or non-financial competing interests to report.

\textbf{Data Availability Statement}
All data generated or analyzed during this study are included in this article.

\end{document}